\documentclass[a4paper,11pt]{amsart}

\usepackage{amsfonts,amssymb}
\usepackage{textcomp, color}

\theoremstyle{plain}

\newtheorem*{thmA}{Theorem A}
\newtheorem*{thmB}{Theorem B}

\newtheorem{thm}{Theorem}[section]
\newtheorem{lem}[thm]{Lemma}
\newtheorem{pro}[thm]{Proposition}
\newtheorem{cor}[thm]{Corollary}

\theoremstyle{definition}

\newtheorem{dfn}[thm]{Definition}

\newtheorem{exa}[thm]{Example}

\numberwithin{equation}{section}

\newcommand{\F}{\mathbb{F}}
\newcommand{\Z}{\mathbb{Z}}
\newcommand{\Q}{\mathbb{Q}}
\newcommand{\N}{\mathbb{N}}

\newcommand{\C}{\mathbb{C}}
\renewcommand{\P}{\mathbb{P}}

\newcommand{\No}{\mathcal{N}}
\newcommand{\MM}{\mathcal{M}}
\newcommand{\WW}{\mathcal{W}}

\begin{document}

\title{Beauville structures in finite $p$-groups}

\author[G.A.\ Fern\'andez-Alcober]{Gustavo A.\ Fern\'andez-Alcober}
\address{Department of Mathematics\\ University of the Basque Country UPV/EHU\\
48080 Bilbao, Spain}
\email{gustavo.fernandez@ehu.eus}

\author[\c{S}.\ G\"ul]{\c{S}\"ukran G\"ul}
\address{Department of Mathematics\\ Middle East Technical University\\
06800 Ankara, Turkey}
\email{gsukran@metu.edu.tr}

\keywords{Beauville groups; finite $p$-groups; Nottingham group\vspace{3pt}}

\thanks{Both authors acknowledge financial support from the Spanish Government, grant MTM2014-53810-C2-2-P, and from the Basque Government, grants IT753-13 and IT974-16.
The first author is also supported by the Spanish Government, grant
MTM2011-28229-C02, and the second author  is supported by T\"UB\.{I}TAK-B\.{I}DEB-2214/A}

\begin{abstract}
We study the existence of (unmixed) Beauville structures in finite $p$-groups, where $p$ is a prime.
First of all, we extend Catanese's characterisation of abelian Beauville groups to finite $p$-groups satisfying certain conditions which are much weaker than commutativity.
This result applies to all known families of $p$-groups with a good behaviour with respect to powers: regular $p$-groups, powerful $p$-groups and more generally potent $p$-groups, and (generalised) $p$-central $p$-groups.
In particular, our characterisation holds for all $p$-groups of order at most $p^{\hspace{.5pt}p}$, which allows us to determine the exact number of Beauville groups of order $p^5$, for $p\ge 5$, and of order $p^6$, for $p\ge 7$.
On the other hand, we determine which quotients of the Nottingham group over $\F_p$ are Beauville groups, for an odd prime $p$.
As a consequence, we give the first explicit infinite family of Beauville $3$-groups, and we show that there are Beauville $3$-groups of order $3^n$ for every
$n\ge 5$.
\end{abstract}

\maketitle

\section{Introduction}

A \emph{Beauville surface\/} (of unmixed type) is a compact complex surface isomorphic to a quotient
$(C_1\times C_2)/G$, where $C_1$ and $C_2$ are algebraic curves of genera at least $2$ and $G$ is a finite group acting freely on $C_1\times C_2$ by holomorphic transformations, in such a way that $C_i/G\cong \P_1(\C)$ and the covering map $C_i\rightarrow C_i/G$ is ramified over three points for $i=1,2$.
Then the group $G$ is said to be a \emph{Beauville group\/}.

The question as to which finite groups are Beauville groups has received considerable attention in recent times.
It can be reformulated in purely group-theoretical terms as follows.
Given two elements $x$ and $y$ of a group $G$, we define
\[
\Sigma(x,y)
=
\bigcup_{g\in G} \,
\Big( \langle x \rangle^g \cup \langle y \rangle^g \cup \langle xy \rangle^g \Big),
\]
that is, the union of the subgroups that belong to the conjugacy classes of $\langle x \rangle$,
$\langle y \rangle$ and $\langle xy \rangle$.
An (unmixed) \emph{Beauville structure\/} for $G$ is then a pair of generating sets $\{x_1,y_1\}$
and $\{x_2,y_2\}$ of $G$ such that $\Sigma(x_1,y_1) \cap \Sigma(x_2,y_2)=1$, and $G$ is a Beauville group if and only if $G$ possesses a Beauville structure.
In particular, Beauville groups are $2$-generator groups.

Research activity around Beauville groups has been very intense since the beginning of this century; see, for example, the recent survey papers
\cite{bos,fai,jon}.
We briefly mention some results that we want to highlight.
In 2000, Catanese \cite{cat} proved that a finite abelian group is a Beauville group if and only if it is isomorphic to $C_n\times C_n$, where  $\gcd(n,6)=1$.
In 2006, Bauer, Catanese, and Grunewald \cite[Conjecture 7.17]{BCG} conjectured that every finite non-abelian simple group other than $A_5$ is a Beauville group.
By using probabilistic methods, Garion, Larsen, and Lubotzky \cite{GLL} showed in 2012 that the conjecture is true if the order of the group is large enough.
Soon afterwards, Guralnick and Malle \cite{GM} gave a complete proof of the conjecture, and Fairbairn, Magaard and Parker \cite{FMP,FMPc}
proved that actually all finite quasisimple groups are Beauville, with the only exceptions of $A_5$ and $SL_2(5)$.

At the other side of the spectrum, if $p$ is a prime, the knowledge about which finite $p$-groups are Beauville groups is very scarce, and is restricted to either groups of small order or with a very simple structure.
Barker, Boston, and Fairbairn \cite{BBF} have determined all Beauville $p$-groups of order at most $p^4$, and have found estimates for the number of Beauville groups of orders $p^5$ and $p^6$.
They have also proved the existence of a non-abelian Beauville $p$-group of order $p^n$ for every $p\ge 5$ and every $n\ge 3$, and have shown that the smallest non-abelian Beauville $p$-groups for $p=2$ and $p=3$ are of order $2^7$ and $3^5$, respectively.
Some of these results from \cite{BBF} rely on computations with the computer algebra system MAGMA.
On the other hand, as a consequence of the main theorem in \cite{BBPV}, there are Beauville $2$-groups of arbitrarily high order.
Finally, let us mention that Stix and Vdovina have shown in \cite[Theorem 3]{SV} that a split metacyclic
$p$-group $G$ is a Beauville group if and only if $p\ge 5$ and $G$ is a semidirect product of two cyclic groups of the same order.

Our goal in this paper is to deepen the knowledge about Beauville $p$-groups.
Notice that Catanese's characterisation of abelian Beauville groups implies that a $2$-generator abelian $p$-group of exponent $p^e$ is a Beauville group if and only if $p\ge 5$ and
$|G^{p^{e-1}}|=p^2$.
In Section 2, we provide a generalisation of this result to a wide class of non-abelian $p$-groups.

\begin{thmA}
Let $G$ be a $2$-generator finite $p$-group of exponent $p^e$, and suppose that $G$ satisfies one of the following conditions:
\begin{enumerate}
\item
For $x,y\in G$, we have $x^{p^{e-1}} = y^{p^{e-1}}$ if and only if $(xy^{-1})^{p^{e-1}} = 1$.
\item
$G$ is a potent $p$-group.
\end{enumerate}
Then $G$ is a Beauville group if and only if $p\ge 5$ and $|G^{p^{e-1}}|\ge p^2$.
\end{thmA}

Recall that a finite $p$-group is potent if either $p>2$ and $\gamma_{p-1}(G)\le G^p$, or $p=2$ and $G'\le G^4$.
They encompass the well-known powerful $p$-groups.
Actually, Theorem A applies to all usual families of $p$-groups with a `nice power structure'.
Observe that (i) is known to hold for regular $p$-groups and for $p$-central $p$-groups.
More generally, we prove that (i) holds for the $p$-groups that we call generalised $p$-central.
These are defined by conditions which are dual to those for potent $p$-groups, namely that $\Omega_1(G)\le Z_{p-2}(G)$ if $p$ is odd and that $\Omega_2(G)\le Z(G)$ if $p=2$.

Since $p$-groups of class less than $p$ are regular, they are also covered by Theorem A.
Thus we can use it to determine in a rather straightforward way whether any given group of order at most $p^{\hspace{0.5pt}p}$ is a Beauville group.
In particular, we determine all Beauville $p$-groups of order $p^5$, for $p\ge 5$, and of order $p^6$, for $p\ge 7$, completing the work in \cite{BBF}.
Theorem A can be similarly applied to obtain all Beauville groups of order $p^7$ for $p\ge 7$, but we have not pursued this task so far.
On the other hand,  since $p$-groups in which $G'$ is cyclic are regular for odd $p$, we can extend the characterization of split metacyclic $p$-groups given in \cite{SV} to all metacyclic $p$-groups.

Unlike in the abelian case, under conditions (i) or (ii) in Theorem A, it may well happen that $|G^{p^{e-1}}|>p^2$; for example, if $G$ is of exponent $p$.
On the other hand, we want to remark that Theorem A is not valid for all finite $p$-groups.
As a matter of fact, we will show that, for each of the implications in Theorem A, there exist infinitely many $p$-groups for which that implication fails.
Even more, we will see that no condition on the order of $G^{p^{e-1}}$ can ensure the existence of Beauville structures in the class of all finite $p$-groups.

\vspace{5pt}

The second main result in our paper deals with the Nottingham group $\No$ over the field $\F_p$, for odd $p$.
This is the group of normalised automorphisms of the ring $\F_p[[t]]$ of formal power series; every $f\in\No$ is completely determined by its image on $t$, which is of the form $f(t)=t+\sum_{i\ge 2} \, a_it^i$.
The group $\No$ is a pro-$p$ group that can be topologically generated by two elements, and so every finite quotient of $\No$ is a candidate to be a Beauville group.
The lattice of normal subgroups of $\No$ is well known.
As a matter of fact, $\No$ is just infinite as an abstract group, i.e.\ every non-trivial normal subgroup of $\No$ is of finite index.
In Section \ref{nott}, for every $1\ne \WW \trianglelefteq \No$ we are able to determine whether the factor group $\No/\WW$ is a Beauville group or not.
It turns out that most quotients of $\No$, but not all, are Beauville groups, and we get the following result.

\begin{thmB}
Let $\No$ be the Nottingham group over $\F_p$, where $p$ is an odd prime, and let
$n_0=2$ or $5$, according as $p>3$ or $p=3$.
Then for every $n\ge n_0$ there exists a quotient of $\No$ of order $p^n$ which is a Beauville group.
\end{thmB}

The $3$-groups in Theorem B constitute the first explicit example which is known of an infinite family of Beauville $3$-groups.
We observe that the existence of infinitely many Beauville $3$-groups follows from either Theorem 2 in \cite{SV} or Theorem 37 in \cite{GJ}, but these results do not provide explicit groups.

\vspace{10pt}

\noindent
\textit{Notation.\/}
If $G$ is a finite $p$-group and $i\ge 0$, we write $\Omega_{\{i\}}(G)$ for the set of all elements of $G$ of order at most $p^i$, and
$\Omega_i(G)$ for the subgroup they generate.
Also, $G^{p^i}$ is the subgroup generated by all powers $g^{p^i}$, as $g$ runs over $G$.
The exponent of $G$, denoted by $\exp G$, is the maximum of the orders of all elements of $G$.
If $g\in G$ is an element of order $p$ and $\lambda\in\F_p$, then $g^{\lambda}$ is understood to be $g^n$, where $n$ is any integer which reduces to $\lambda$ modulo $p$.
The rest of the notation is standard in group theory.

\section{Finite $p$-groups with a nice power structure}

Our goal in this section is to prove Theorem A in the introduction, which generalises Catanese's criterion for abelian Beauville groups to all classes of finite
$p$-groups with a `nice power structure' that have been considered in the literature.
Then we give a number of applications of this result, like the determination of all Beauville groups of order $p^5$ for $p\ge 5$ or of order $p^6$ for $p\ge 7$, the characterisation of Beauville metacyclic $p$-groups, and also that of Beauville $p$-groups of maximal class of order at most $p^{\hspace{0.4pt}p}$.

We begin by introducing a property which is essential to our result, and by studying which of the best known families of finite $p$-groups satisfy this property.
Let $G$ be a finite $p$-group, and let $i\ge 1$ be an integer.
Following Xu \cite{xu}, we say that $G$ is \emph{semi-$p^i$-abelian\/} if the following condition holds:
\begin{equation}
\label{semi-pi-abelian}
x^{p^i} = y^{p^i}
\quad
\text{if and only if}
\quad
(xy^{-1})^{p^i} = 1.
\end{equation}
If $G$ is semi-$p^i$-abelian, then we have \cite[Lemma 1]{xu}:
\begin{enumerate}
\item[(SA1)]
$\Omega_i(G)=\{x\in G\mid x^{p^i}=1\}$.
\item[(SA2)]
The cardinality of the set $\{x^{p^i}\mid x\in G\}$ coincides with the index $|G:\Omega_i(G)|$.
\end{enumerate}
If $G$ is semi-$p^i$-abelian for every $i\ge 1$, then $G$ is called \emph{strongly semi-$p$-abelian\/}.

For example, regular $p$-groups are strongly semi-$p$-abelian \cite[Theorem 3.14]{suz2}.
On the other hand, a powerful $p$-group of exponent $p^e$ is semi-$p^{e-1}$-abelian, as follows from Lemma 3 of \cite{fer}.
Also, Xu proved that, for odd $p$, any finite $p$-group satisfying $\Omega_1(\gamma_{p-1}(G))\le Z(G)$ is strongly semi-$p$-abelian.
This applies in particular to the case $\Omega_1(G)\le Z(G)$, i.e.\ to $p$-central $p$-groups.
Next we consider two other families of $p$-groups, which are generalisations of $p$-central $p$-groups and powerful $p$-groups, respectively, and we study their behaviour with respect to property (\ref{semi-pi-abelian}).

\begin{dfn}
Let $G$ be a finite $p$-group.
Then:
\begin{enumerate}
\item
$G$ is \emph{generalised $p$-central\/} if $\Omega_1(G)\le Z_{p-2}(G)$ for odd $p$, and if
$\Omega_2(G)\le Z(G)$ if $p=2$.
\item
$G$ is \emph{potent\/} if $\gamma_{p-1}(G)\le G^p$ for odd $p$, and if $G'\le G^4$ for $p=2$. 
\end{enumerate}
\end{dfn}

Observe the duality between the two cases.
Clearly, when $p=2$ or $3$ these concepts coincide with $p$-central $p$-groups and with powerful $p$-groups.
Generalised $p$-central $p$-groups have been studied by Gonz\'alez-S\'anchez and Weigel \cite{GW}, and potent $p$-groups by Arganbright \cite{arg}, Lubotzky and Mann \cite{LM}, and Gonz\'alez-S\'anchez and Jaikin-Zapirain \cite{GJ2}.
If $G$ is generalised $p$-central then, by \cite[Theorem B]{GW}, we have
\[
\Omega_i(G) = \{ g\in G \mid g^{p^i}=1 \},
\quad
\text{for all $i\ge 1$,}
\]
and if $G$ is potent then, by \cite[Theorem 2]{arg} and \cite[Proposition 1.7]{LM}, we have
\[
G^{p^i} = \{ g^{p^i} \mid g\in G \},
\quad
\text{for all $i\ge 1$.}
\]

\begin{thm}
A generalised $p$-central $p$-group is strongly semi-$p$-abelian.
\end{thm}

\begin{proof}
Given $x,y\in G$, we prove that $x^{p^i}=y^{p^i}$ if and only if $(xy^{-1})^{p^i}=1$ by induction on $i\ge 1$.
We first deal with the case $i=1$, by induction on the order of $G$.
Thus the result holds in any proper subgroup of $G$, and we may assume that $G=\langle x,y \rangle$.

Let us see that $x^p=y^p$ implies $(xy^{-1})^p=1$.
We have $(x^y)^p=(x^p)^y=x^p$.
Since $\langle x,x^y \rangle<G$ unless $G$ is cyclic, it follows that $[x,y]^p=(x^{-1}x^y)^p=1$.
Consequently $G'\le \Omega_1(G)$.
If $p$ is odd, then $G'\le Z_{p-2}(G)$, and $\gamma_p(G)=[G',G,\overset{p-2}{\ldots},G]=1$.
Thus $G$ is regular and $(xy^{-1})^p=1$, as desired.
If $p=2$ then, since $x^2=y^2$, we have $G/G'=C/G'\times D/G'$ with $C/G'$ and $D/G'$ cyclic, and $D/G'\cong C_2$.
Then $D\le \Omega_2(G)\le Z(G)$ and $G/Z(G)$ is cyclic.
Thus $G$ is abelian and the result is valid also in this case.

Now we prove that $(xy^{-1})^p=1$ implies $x^p=y^p$.
We have $(xy^{-1})^p=1=(y^{-1}x)^p$ and, by the previous paragraph, $[x,y^{-1}]^p=1$.
Consequently again $G'\le\Omega_1(G)$ and then, as above, $G$ is regular (i.e.\ abelian if $p=2$) and the result holds.

Finally, we consider the case $i>1$.
As shown above, we have $x^{p^i}=y^{p^i}$ if and only if $(x^{p^{i-1}}y^{-p^{i-1}})^p=1$, which is in turn equivalent to
$x^{p^{i-1}} \Omega_1(G) = y^{p^{i-1}} \Omega_1(G)$.
Now, $G/\Omega_1(G)$ is again a generalised $p$-central $p$-group, by  \cite[Theorem B]{GW}.
By the induction hypothesis, the last equality is tantamount to $(xy^{-1})^{p^{i-1}}\in \Omega_1(G)$, and this means exactly that $(xy^{-1})^{p^i}=1$.
\end{proof}

However, potent $p$-groups are not in general strongly semi-$p$-abelian, even not semi-$p^{e-1}$-abelian as powerful $p$-groups, given that $\exp G=p^e$.

\begin{exa}
Let $p>3$ be a prime and let
\[
A
=
\langle a_1\rangle \times \dots\times\langle a_p\rangle
\cong
C_p\times\overset{p-2}{\dots}\times C_p\times C_{p^2}\times C_{p^2}.
\]
We define an automorphism $\alpha$ of $A$ by means of $\alpha(a_i)=a_ia_{i+1}$, for $i=1,\ldots,p-3$, $\alpha(a_{p-2})=a_{p-2}a_{p-1}^p$,
$\alpha(a_{p-1})=a_{p-1}a_p$, and $\alpha(a_p)=a_p$.
Then $\alpha$ is of order $p^2$, and we can construct a semidirect product $G=\langle b\rangle \ltimes A$, where $b$ is of order $p^2$ and acts on $A$ via
$\alpha$.
One readily checks that $\gamma_{p-1}(G)=\langle a_{p-1}^p, a_p^p \rangle \leq G^p$, and hence $G$ is a potent $p$-group.
Also, we have $\exp G=p^2$.
However, $G$ is not semi-$p$-abelian, since $x=ba_1$ and $y=ba_p$ satisfy that $x^p=b^pa_p^p=y^p$, but $(xy^{-1})^p=a_p^{-p}\ne 1$.
\end{exa}

After this preliminary analysis of the finite $p$-groups satisfying condition (\ref{semi-pi-abelian}), we proceed to the proof of Theorem A.
We start with a result of a more general nature that can be used to prove the non-existence of Beauville structures.

\begin{pro}
\label{negative}
Let $G$ be a $2$-generator finite $p$-group of exponent $p^e$, and suppose that:
\begin{enumerate}
\item
$\Omega_{\{e-1\}}(G)$ is contained in the union of two maximal subgroups of $G$.
\item
$|G^{p^{e-1}}|=p$.
\end{enumerate}
Then $G$ is not a Beauville group.
\end{pro}

\begin{proof}
We argue by way of contradiction.
Suppose $\{x_1,y_1\}$ and $\{x_2,y_2\}$ are two systems of generators of $G$ such that $\Sigma(x_1,y_1)\cap \Sigma(x_2,y_2)=1$.
Since no two of the elements $x_1$, $y_1$ and $x_1y_1$ can lie in the same maximal subgroup of $G$, it follows from (i) that one of these elements, say $x_1$, is of order $p^e$.
Similarly, we may assume that the order of $x_2$ is also $p^e$.
Since $G^{p^{e-1}}$ is of order $p$, we conclude that $\langle x_1^{p^{e-1}} \rangle=\langle x_2^{p^{e-1}} \rangle$, which is a contradiction.
\end{proof}

As we will see at the end of the paper, we cannot relax condition (i) in the last proposition, since there are examples of groups $G$ in which
$\Omega_{\{e-1\}}(G)$ is contained in the union of three maximal subgroups, and which are Beauville groups even if $G^{p^{e-1}}$ is of order $p$.

The following result is an extended version of Theorem A.

\begin{thm}
\label{characterisation general}
Let $G$ be a $2$-generator finite $p$-group of exponent $p^e$ which satisfies one of the following conditions:
\begin{enumerate}
\item
$G$ is semi-$p^{e-1}$-abelian, i.e.\ for every $x,y\in G$, we have $x^{p^{e-1}}=y^{p^{e-1}}$ if and only if $(xy^{-1})^{p^{e-1}}=1$.
\item
$G$ is potent.
\end{enumerate}
Then $G$ is a Beauville group if and only if $p\ge 5$ and $|G^{p^{e-1}}|\ge p^2$.
If that is the case, then every lift of a Beauville structure of $G/\Phi(G)$ is a Beauville structure of $G$.
\end{thm}

\begin{proof}
Before proceeding, notice that in any case $\Omega_{e-1}(G)$ coincides with the set of all elements of $G$ of order at most $p^{e-1}$, and the cardinality of
\[
X = \{ g^{p^{e-1}} \mid g\in G \}
\]
coincides with the index $|G:\Omega_{e-1}(G)|$.
This follows from (SA1) and (SA2) if $G$ is semi-$p^{e-1}$-abelian, and from Theorem 6.6 of \cite{GJ2} if $G$ is potent.
Next, we claim that the condition $|G^{p^{e-1}}|\ge p^2$ implies that $\Omega_{e-1}(G)$ is contained in $\Phi(G)$.
Since $\Phi(G)=G'G^p\le G'\Omega_{e-1}(G)$, we have
\begin{equation}
\label{chain}
|G/\Omega_{e-1}(G):(G/\Omega_{e-1}(G))'|
=
|G:G'\Omega_{e-1}(G)|
\le
|G:\Phi(G)| = p^2.
\end{equation}
If $|G/\Omega_{e-1}(G):(G/\Omega_{e-1}(G))'|\le p$, then the quotient $G/\Omega_{e-1}(G)$ is cyclic, and so it has order at most $p$.
It follows that $|X|\le p$, and so also $|G^{p^{e-1}}|\le p$, which is a contradiction.
Thus $|G/\Omega_{e-1}(G):(G/\Omega_{e-1}(G))'|\ge p^2$, and this, together with (\ref{chain}), yields that $\Omega_{e-1}(G)\le \Phi(G)$.
This proves the claim and, as a consequence, all elements outside $\Phi(G)$ are of order $p^e$.
 
Now we show that $G$ is a Beauville group if $p\ge 5$ and $|G^{p^{e-1}}|\ge p^2$.
Since $p\ge 5$, the elementary abelian group $G/\Phi(G)$ is a Beauville group.
Let us see that every lift of a Beauville structure of $G/\Phi(G)$ is a Beauville structure of $G$.
If we use the bar notation in $G/\Phi(G)$, it suffices to show that, given two elements
$x,y\in G\smallsetminus \Phi(G)$, the condition $\langle \overline x \rangle \cap \langle \overline y \rangle=\overline 1$ implies that $\langle x \rangle \cap \langle y \rangle=1$.
Suppose for a contradiction that $\langle x \rangle \cap \langle y \rangle \ne 1$, i.e.\ that
$x^{p^{e-1}}=y^{ip^{e-1}}$ for some integer $i$ not divisible by $p$.

Now we deal separately with the two conditions (i) and (ii).
If (i) holds then $xy^{-i}\in \Omega_{e-1}(G)$, and consequently $\langle \overline x \rangle=\langle \overline y \rangle$ in $G/\Phi(G)$, which is a contradiction.
On the other hand, if $G$ is potent then we consider a maximal subgroup $M$ of $G$ containing $x$.
Since $\Omega_{e-1}(G)\le M$, we have
\[
|M:M^{p^{e-1}}| = |\Omega_{e-1}(M)| = |\Omega_{e-1}(G)| = |G:G^{p^{e-1}}|,
\]
by \cite[Theorem 6.6]{GJ2}.
As a consequence, $|G^{p^{e-1}}:M^{p^{e-1}}|=p$.
Now the quotient group $G/M^{p^{e-1}}$ is a potent group which can be generated by two elements of order at most $p^{e-1}$, namely the images of $x$ and $y$. 
Again by Theorem 6.6 of \cite{GJ2}, it follows that $\exp G/M^{p^{e-1}}\le p^{e-1}$ and $G^{p^{e-1}}\le M^{p^{e-1}}$.
This contradiction completes the proof of the first implication in the statement of the theorem.

Let us now prove the converse.
Since $\exp G=p^e$ and $\Omega_{\{e-1\}}(G)$ is a proper subgroup of $G$, it follows from
Proposition \ref{negative} that we only need to prove that $G$ has no Beauville structure if $p=2$ or $3$, provided that $|G^{p^{e-1}}|\ge p^2$.
As shown above, all elements of $G\smallsetminus \Phi(G)$ are of order $p^e$.
Also, if $G$ is potent then $G$ is powerful and consequently semi-$p^{e-1}$-abelian.
Consequently, it suffices to prove the result in case (i).
We are going to show that any Beauville structure of $G$ induces, by passing to the quotient, a Beauville structure in $G/G^p$.
However, if $p=2$ then $G/G^2$ is abelian of order $4$, and if $p=3$ then $G/G^3$ is of order at most $3^3$ by \cite[14.2.3]{rob}.
In both cases, $G/G^p$ does not have a Beauville structure.
 	
So let us see that every Beauville structure of $G$ is inherited by $G/G^p$.
To this end, we see that, given $x,y\in G\smallsetminus \Phi(G)$, the condition
$\langle x \rangle \cap \langle y \rangle=1$ implies that
$\langle \overline x \rangle \cap \langle \overline y \rangle=\overline 1$ in $G/G^p$.
If this is false, then for some $i$ not divisible by $p$ we have $xy^{-i}\in G^p$, and consequently
$(xy^{-i})^{p^{e-1}}=1$.
By (i), we get $x^{p^{e-1}}=y^{ip^{e-1}}$.
This is a contradiction, since both $x$ and $y$ are of order $p^e$ and $\langle x \rangle \cap \langle y \rangle=1$.
\end{proof}

\begin{cor}
\label{cor:thm A for all good families}
Let $G$ be a finite $p$-group, and suppose that $G$ belongs to one of the following families:
\begin{enumerate}
\item
Regular $p$-groups, and in particular groups of order at most $p^{\hspace{0.4pt}p}$.
\item
Potent $p$-groups, and in particular powerful $p$-groups.
\item
Generalised $p$-central $p$-groups, and in particular $p$-central $p$-groups.
\item
For odd $p$, groups satisfying the condition $\Omega_1(\gamma_{p-1}(G))\le Z(G)$.
\end{enumerate}
Then $G$ is a Beauville group if and only if $p\ge 5$ and $|G^{p^{e-1}}|\ge p^2$, where $\exp G=p^e$.
\end{cor}

As we see in our next result, under the hypotheses of Theorem A, the condition $|G^{p^{e-1}}|\ge p^2$ is easy to detect if we know a reasonably good presentation of $G$.

\begin{pro}
\label{easy detecting}
Let $G=\langle a,b \rangle$ be a finite $p$-group of exponent $p^e$ which is either semi-$p^{e-1}$-abelian or potent, where $p\ge 5$.
Then $|G^{p^{e-1}}|\ge p^2$ if and only if the two subgroups $\langle a^{p^{e-1}} \rangle$ and
$\langle b^{p^{e-1}} \rangle$ are different and non-trivial.
\end{pro}

\begin{proof}
We only need to prove the `only if' part.
Thus we assume that $|G^{p^{e-1}}|\ge p^2$.
From the proof of Theorem \ref{characterisation general}, we have $\Omega_{e-1}(G)\le \Phi(G)$, and so both $\langle a^{p^{e-1}} \rangle$ and
$\langle b^{p^{e-1}} \rangle$ are non-trivial.
Now, since $p\ge 5$, there is a Beauville structure with two systems of generators $\{a\Phi(G),y_1\Phi(G)\}$ and $\{b\Phi(G),y_2\Phi(G)\}$.
Again by Theorem \ref{characterisation general}, the lift given by $\{a,y_1\}$ and $\{b,y_2\}$ is a Beauville structure of $G$.
Thus $\langle a \rangle \cap \langle b \rangle=1$, and consequently $\langle a^{p^{e-1}} \rangle \ne \langle b^{p^{e-1}} \rangle$.
This completes the proof.
\end{proof}

\begin{exa}
In \cite[Conjecture 25]{BBF}, the authors conjecture that the two non-isomorphic groups of order $p^5$ under the names $H_3$ and $H_4$ in that paper are Beauville for $p\ge 5$ and, as a consequence, that there are exactly $p+10$ Beauville groups of order $p^5$ for $p\ge 5$.
This can be confirmed in a straightforward way by using Proposition \ref{easy detecting}.
Indeed, both isomorphism types can be described in the form
\begin{gather*}
G = \langle a, b, c \mid a^{p^2}=b^{p^2}=c^p=[b,c]=1,\ [a,b]=c,\ [a,c]=b^{rp} \rangle,
\end{gather*}
where $r$ is not divisible by $p$.
Observe that $G=\langle a,b \rangle$.
Since $\exp G=p^2$ and $\langle a^p \rangle \ne \langle b^p \rangle$ are non-trivial, we conclude that $G$ is a Beauville group.
\end{exa}

For groups of class less than $p$ (so in particular for groups of order at most $p^{\hspace{0.4pt}p}$), we can further simplify the determination of whether the group is Beauville or not by using the Lazard Correspondence.
Recall that the Lazard Correspondence uses the Baker-Campbell-Hausdorff formula to establish a one-to-one correspondence between finite $p$-groups of class less than $p$ and nilpotent Lie rings of $p$-power order of class less than $p$ (see \cite[Section 10.2]{khu}).
The underlying set for both the group and the Lie ring is the same, and it turns out that the $n$th power of an element in the group coincides with its $n$th multiple in the Lie ring.
Thus if $G=\langle a,b \rangle$, we can check the conditions in Proposition \ref{easy detecting} by working in the Lie ring instead of in the group, i.e.\ we have to check whether $\langle p^{e-1}a \rangle$ and $\langle p^{e-1}b \rangle$ are different and non-trivial.
This is particularly interesting for $p$-groups of small order, since their classification relies in classifying first nilpotent Lie rings of the same order and then applying the Lazard correspondence.
For example, this is the procedure followed in \cite{NOV} and \cite{OV} to determine all groups of orders $p^6$ and $p^7$ for $p\ge 7$.
By using the presentations of the nilpotent Lie rings of order $p^6$ provided in \cite{vau}, we have obtained that the number of Beauville groups of order $p^6$ is
\[
4p+20+4\gcd(p-1,3)+\gcd(p-1,4)
\]
for $p\ge 7$.
Since the total number of $2$-generator groups of order $p^6$ is
\[
10p+62+14\gcd(p-1,3)+7\gcd(p-1,4)+2\gcd(p-1,5),
\]
it follows that the ratio between the number of Beauville groups and the number of all $2$-generator groups of order $p^6$ tends to $2/5$ as $p\to\infty$.
Note that in \cite{BBF}, the authors could only say that this limit is smaller than $1$, by finding $p-1$ non-Beauville $2$-generator groups of order $p^6$.

As an illustration of our method, let us consider the following two nilpotent Lie rings of order $p^6$ taken from
\cite{vau}:
\[
L_1 = \langle a,b \mid p^2a, pb-[b,a,a], \text{$p$-class 3} \rangle
\]
and
\begin{multline*}
L_2 = \langle a,b \mid pa-[b,a,a,a], pb-[b,a,a,a], [b,a,a,b],
\\
[b,a,b,b]+[b,a,a,a], \text{$p$-class 4} \rangle.
\end{multline*}
Clearly, in both $L_1$ and $L_2$, the elements $a$ and $b$ are of order $p^2$, which is the exponent of the additive group of the Lie ring.
In the first case, $\langle pa \rangle$ and $\langle pb \rangle$ are different subgroups, while in the second case they are equal.
Thus $L_1$ gives rise to a Beauville group under the Lazard correspondence, while $L_2$ does not.

It would be equally possible to calculate the exact number of Beauville groups of order $p^7$, for $p\ge 7$, but we have not pursued that task.

\vspace{10pt}

As a final application of Theorem A, we extend the characterisation given by Stix and Vdovina \cite[Theorem 3]{SV} of split metacyclic Beauville $p$-groups to all metacyclic $p$-groups.
Our proof is self-contained and does not rely on the result of \cite{SV}.

\begin{cor}
A metacyclic $p$-group $G$ is a Beauville group if and only if $p\ge 5$ and $G$ is a semidirect product of two cyclic groups of the same order.
\end{cor}

\begin{proof}
Let $G=\langle a,b \rangle$ with $\langle a \rangle \trianglelefteq G$.
Assume first that $p$ is odd, and let $\exp G=p^e$.
Then $G'$ is cyclic and $G$ is regular by \cite[3.13]{suz2}.
Thus $G$ is semi-$p^{e-1}$-abelian and then, by Proposition \ref{easy detecting}, $G$ is a Beauville group if and only if $p\ge 5$ and the subgroups $\langle a^{p^{e-1}} \rangle$ and $\langle b^{p^{e-1}} \rangle$ are non-trivial and different.
This means that $G=\langle b \rangle \ltimes \langle a \rangle$ is a semidirect product with $a$ and $b$ of the same order.

Now we consider the case $p=2$.
We have to prove that $G$ is not a Beauville group. 
If $G'\le G^4$ then $G$ is powerful, and the result follows from Corollary \ref{cor:thm A for all good families}.
Thus we assume that $G'$ is not contained in $G^4$, i.e.\ that $G'=\langle a^2 \rangle$.
We claim that, for every set $\{x,y\}$ of generators, we have $bG'\subseteq \Sigma(x,y)$.
This proves that $G$ is not a Beauville group also in this case.
Since $G$ has only three maximal subgroups, we may assume that $x\in \langle b \rangle G^2 \smallsetminus G^2$.
Since $G^2=\langle b^2 \rangle G'$, we can write $x=b^iw$ with $i$ odd and $w\in G'$.
Then there is a power of $x$ of the form $x^*=bw^*$, for some $w^*\in G'$.
Now observe that
\[
G' = \langle [x^*,a] \rangle = \{ [x^*,a]^i \mid i\in\N \} = \{ [x^*,a^i] \mid i\in\N \},
\]
and consequently the conjugacy class of $x^*$ equals $x^*G'=bG'$.
This proves that $bG'$ is contained in $\Sigma(x,y)$, as desired.
\end{proof}

\begin{cor}
Let $G$ be $p$-group of maximal class of order at most $p^{\hspace{0.4pt}p}$.
Then $G$ is a Beauville group if and only if $p\ge 5$ and $\exp G=p$.
\end{cor}

\begin{proof}
By a result of Blackburn (see \cite[Theorem 3.2]{bla} or \cite[Lemma 14.14]{hup}), we have $\exp G/Z(G)=p$.
Thus $|G^p|\le p$ and $\exp G=p$ or $p^2$.
By Theorem A, $G$ is a Beauville group if and only if $p\ge 5$ and $\exp G=p$.
\end{proof}

The assumptions (i) or (ii) are essential in Theorem \ref{characterisation general}.
Indeed, for a general finite $p$-group $G$, the condition that $|G^{p^{e-1}}|\ge p^2$ is neither sufficient nor necessary for $G$ to be a Beauville group.
We show this in Corollary \ref{cor:both implications fail}, by using quotients of some infinite pro-$p$ groups that we define now.

Let $k\ge 1$ be a fixed integer, and consider the ring of integers $R$ of the cyclotomic field $\Q_p(\zeta)$, where $\Q_p$ is the field of $p$-adic numbers and
$\zeta$ is a primitive $p^k$th root of unity.
Then $R=\Z_p[\zeta]$ is a discrete valuation ring and a free $\Z_p$-module of rank $p^{k-1}(p-1)$.
Also, the element $\zeta-1$ is a uniformizer, and we have
\begin{equation}
\label{eqn:p ideal}
(p)=(\zeta-1)^{p^{k-1}(p-1)}.
\end{equation}
Multiplication by $\zeta$ defines an automorphism of order $p^k$ of the additive group of $R$, which can be used to construct a split extension of $R$ by $C_{p^k}$.
In order to avoid mixing additive and multiplicative notation, we consider a multiplicative copy $A$ of $R$, via an isomorphism $\varphi:A\rightarrow R$.
If $C=\langle t \rangle$ is a cyclic group of order $p^k$, then we define $P_k=C\ltimes A$, where the action of $t$ on $A$ corresponds under $\varphi$ to multiplication by $\zeta$, that is, $\varphi(a^t)=\zeta\varphi(a)$ for all $a\in A$.
Observe that $P_k$ is a $2$-generator pro-$p$ group, topologically generated by $t$ and by $a_1=\varphi^{-1}(1)$.
Also, we have $P_k'=[A,t]$, which corresponds to the ideal $(\zeta-1)$ of $R$ under $\varphi$.
More generally, the lower central series of $P_k$ consists of the subgroups $[A,t,\ldots,t]$, and the action of $C$ on $A$ is uniserial.
In particular, if $k=1$ then we get the only infinite pro-$p$ group of maximal class.

As we next see, the pro-$p$ groups $P_k$ are a source of infinitely many Beauville $p$-groups.

\begin{thm}
\label{thm:quotients of Pk}
Let $p\ge 5$ be a prime and let $k\ge 1$ be an integer.
If $N$ is a normal subgroup of $P_k$ of finite index and $N\le A^p$, then the factor group $P_k/N$ is a Beauville group.
\end{thm}

\begin{proof}
By \cite[Theorem 1.17]{DDSMS}, $N$ is open in $P_k$, and consequently, $P_k/N$ is a $2$-generator $p$-group.
For every $a\in A$, we have
\begin{equation}
\label{eqn:power of ta}
(ta)^{p^k} = t^{p^k} a^{\sum_{i=0}^{p^k-1} \, t^i}.
\end{equation}
Since $\zeta$ is a primitive $p^k$th root of unity and $t^{p^k}=1$, it follows that $(ta)^{p^k}=1$.
Now the image of $ta$ in $P_k/A$ is of order $p^k$, and consequently $taN$ is of order $p^k$ in $P_k/N$ as well.

Let $a,b\in A$ and assume that the subgroups generated by $taN$ and $tbN$ have non-trivial intersection.
Then these subgroups have the same $p^{k-1}$st power, and it readily follows that
$(taN)^{p^{k-1}}=(tbN)^{p^{k-1}}$.
By a calculation similar to (\ref{eqn:power of ta}), we get
\[
a^{\sum_{i=0}^{p^{k-1}-1} \, t^i} \equiv b^{\sum_{i=0}^{p^{k-1}-1} \, t^i} \pmod N,
\]
and then the same congruence holds modulo $A^p$.
It follows that
\begin{equation}
\label{eqn:in p}
\Big( \sum_{i=0}^{p^{k-1}-1} \, \zeta^i \Big) (\varphi(a)-\varphi(b)) \equiv 0 \pmod p
\end{equation}
in $R$.
Now in the polynomial ring $\F_p[X]$ we have
\[
\sum_{i=0}^{p^{k-1}-1} \, X^i  = (X-1)^{p^{k-1}-1},
\]
and consequently
\[
\sum_{i=0}^{p^{k-1}-1} \, \zeta^i  \equiv (\zeta-1)^{p^{k-1}-1} \pmod{p}.
\]
If we replace this into (\ref{eqn:in p}) and use (\ref{eqn:p ideal}), we get
\[
\varphi(a)-\varphi(b) \in (\zeta-1)^{p^k-2p^{k-1}+1}.
\]
In particular, $\varphi(a)-\varphi(b)\in (\zeta-1)$ or, what is the same, $a\equiv b\pmod{P_k'}$.

Now, by the previous paragraph, if $x,y\in\{t,ta_1,\ldots,ta_1^{p-1}\}$ and $x\ne y$, then $\langle xN \langle \cap \langle yN \rangle = 1$.
Since $xN$ and $yN$ are generators of $P_k/N$ and $p\ge 5$, we conclude that $P_k/N$ is a Beauville group.
\end{proof}

\begin{cor}
\label{cor:both implications fail}
Let $p\ge 5$ be a prime.
Then, for each of the implications in the criterion for Beauville groups given in Theorem A, there exist infinitely many $2$-generator $p$-groups (and even infinitely many $p$-groups of maximal class) for which the implication fails.
\end{cor}

\begin{proof}
Consider arbitrary integers  $e>k\ge 1$, and let $N$ be a normal subgroup of $P_k$ such that $|A^{p^{e-1}}:N|=p$.
By Theorem \ref{thm:quotients of Pk}, $G=P_k/N$ is a Beauville group.
In the proof of that theorem, we have seen that all elements of the form $taN$ with $a\in A$ are of order $p^k$.
It readily follows that every element of $P_k\smallsetminus A$ is of order at most $p^k$ when passing to $G$.
Thus $\exp G=p^e$ and $G^{p^{e-1}}=A^{p^{e-1}}/N$ is of order $p$.
This shows that the `only if' part of Theorem A fails for $G$.

Let us construct a family of groups for which the `if' part fails.
Take again $e>k\ge 1$, and consider now $N\trianglelefteq P_k$ lying between $A^{p^{e-1}}$ and $A^{p^e}$, and such that
$|A^{p^{e-1}}:N|=p^m\ge p^2$.
Let $L\trianglelefteq P_k$ be an intermediate subgroup between $N$ and $A^{p^{e-1}}$ such that $|L:N|=p$.
Thus $L/N\le Z(P_k/N)$.
Let us write $H=A/N$ and $Z=L/N$.
By using the theory of cyclic extensions \cite[Section 3.7]{zas}, we can get a new group $G=\langle u,H \rangle$, where the action of $u$ on $H$ is again the one induced by multiplication by $\zeta$, but $u^{p^k}\in Z\smallsetminus 1$.
Observe that $G$ is a $2$-generator group.
A calculation as in (\ref{eqn:power of ta}) shows that every $x\in uH$ is now of order $p^{k+1}$, and $\langle x^{p^k} \rangle=Z$.
Also, all elements of $G\smallsetminus H$ are of order at most $p^{k+1}$.
Then $\exp G=p^e$ and $|G^{p^{e-1}}|=p^m$.
Now, any set $\{x,y\}$ of generators of $G$ must contain an element, say $x$, outside the maximal subgroup $\langle u^p \rangle H$, and then a power of
$x$ will be in $uH$.
Consequently, $Z\subseteq \Sigma(x,y)$ and $G$ cannot be a Beauville group.
\end{proof}

We end this section by showing that it is not possible to find a variation of Theorem A which ensures the existence of Beauville structures in an arbitrary finite $p$-group, even if we strengthen the requirement on the size of $G^{p^{e-1}}$.
Indeed, for every power of $p$ there are non-Beauville $p$-groups for which the order of $G^{p^{e-1}}$ is exactly that power.

\begin{cor}
For every prime $p\ge 5$, and positive integer $m$, there exists a $2$-generator $p$-group $G$ such that:
\begin{enumerate}
\item
If $\exp G=p^e$ then $|G^{p^{e-1}}|=p^m$.
\item
$G$ is not a Beauville $p$-group.
\end{enumerate}
\end{cor}

\begin{proof}
Let $G=P_k/N$ be as in the second part of the proof of the last corollary.
Since $\log_p |A^{p^{e-1}}:A^{p^e}| = p^{k-1}(p-1)$, we can make $|A^{p^{e-1}}:N|$ as large as we want by taking $k$ big enough.
This gives the desired groups.
\end{proof}

\section{Quotients of the Nottingham group over $\F_p$ for odd $p$}
\label{nott}

In this section, given an odd prime $p$, we determine which (finite) quotients of the Nottingham group $\No$ over $\F_p$ are Beauville groups.
One of the consequences of our study will be that, in the case $p=3$, there are quotients of $\No$ which are Beauville groups of order $3^n$,
for every $n\ge 5$.
This provides the first explicit infinite family of Beauville $3$-groups in the literature.

Given a fixed positive integer $k$, the automorphisms $f\in\No$ such that $f(t)=t+\sum_{i\ge k+1} \, a_it^i$ form
an open normal subgroup $\No_k$ of $\No$.
Observe that $|\No_k:\No_{k+1}|=p$ for all $k\ge 1$, and that the condition $f\equiv g \pmod{\No_k}$ in the group
$\No$ is equivalent to the congruence $f(t)\equiv g(t) \pmod{t^{k+1}}$ in the ring $\F_p[[t]]$.
By Theorem 1.3 of \cite{klo2}, every non-trivial abstract (i.e.\ not necessarily closed) normal subgroup $\WW$ of $\No$ satisfies either
$\No_{k+1}\le \WW \le \No_k$ for some $k\ge 1$ or $\No_{k+2}\le \WW \le \No_k$ for some $k\equiv 1\pmod p$.
In particular, $\WW$ is of finite index in $\No$.

We recall the following formulas for $p$th powers and commutators of the subgroups $\No_k$
(see \cite{cam}, Theorem 6 and Theorem 2, respectively.)
In the case of powers, we have
\begin{equation}
\label{powers Ni}
\No_k^p = \No_{pk+r},
\quad
\text{where $0\le r\le p-1$ is the residue of $k$ modulo $p$.}
\end{equation}
On the other hand, we have the commutator formula
\begin{equation}
\label{comms Ni}
[\No_k,\No_{\ell}]
=
\begin{cases}
\No_{k+\ell},
&
\text{if $k\not\equiv \ell\pmod p$,}
\\
\No_{k+\ell+1},
&
\text{if $k\equiv \ell\pmod p$.}
\end{cases}
\end{equation}
It follows that $\No'=\No_3$ and, more generally, $\gamma_i(\No)=\No_{s(i)}$, where $s(i)=i+1+\lfloor (i-2)/(p-1) \rfloor$.
As a consequence, $|\gamma_i(\No):\gamma_{i+1}(\No)|\le p^2$, and we have `diamonds' of order $p^2$ if and only if $i$ is of the form $i=k(p-1)+1$ for some $k\ge 0$.
In other words, the diamonds in the lower central series of $\No$ correspond to quotients $\No_{kp+1}/\No_{kp+3}$.
Another consequence of (\ref{comms Ni}) is the following: if
$f\in\No_k\smallsetminus \No_{k+1}$ and $g\in\No_{\ell}\smallsetminus \No_{\ell+1}$ with $k\not\equiv\ell \pmod p$, then
$[f,g]\in \No_{k+\ell}\smallsetminus \No_{k+\ell+1}$.

On the other hand, every non-trivial normal subgroup $\WW$ of $\No$ is either a term of the series $\{\No_k\}$ or it is an intermediate subgroup in a diamond, i.e. we have $\No_{kp+3}<\WW<\No_{kp+1}$ for some $k\ge 0$.
In the latter case, there are $p+1$ subgroups for every $k$; if $e$ and $f$ are the automorphisms defined via
$e(t)=t+t^{kp+2}$ and $f=t+t^{kp+3}$, then these subgroups are $\langle ef^i , \No_{kp+3} \rangle$, where $i=0,\ldots,p-1$, together with
$\No_{kp+2}=\langle f , \No_{kp+3} \rangle$.

In the following, we write $z_m$ for the number $p^m+p^{m-1}+\cdots+p+2$, for every $m\ge 1$.
We extend this definition to the case $m=0$ by putting $z_0=2$.
Observe that $\No_{z_m-1}/\No_{z_m+1}$ is a diamond for all $m\ge 0$; we refer to these as
\emph{distinguished diamonds\/}.
By (\ref{powers Ni}), we have
\begin{equation}
\label{powers Nzm}
\No_{z_m+1}^{p^n} \subseteq \No_{z_{m+n}+1},
\quad
\text{for all $m,n\ge 0$.}
\end{equation}

Our approach to the determination of Beauville structures in quotients of the Nottingham group is based on the analysis of the specific quotients of the form $\No/\No_{z_m+1}$, i.e.\ when we factor out $\No$ at the bottom of a distinguished diamond.
To this purpose, it is fundamental to control the $p^m$th powers of elements outside $\No'$ (which are potential generators) in the factor group $\No/\No_{z_m+1}$.
We need some lemmas.

\begin{lem}
\label{all_elements}
Let $f\in\No_{z_k-1}$ and $g \in \No_{z_k+1}$, where $k\ge 0$.
Then, for every $\ell\ge 1$ we have
\[
(fg)^{p^{\ell}} \equiv f^{p^{\ell}} \pmod{ \No_{z_{k+\ell}+1}}.
\]
\end{lem}

\begin{proof}
By the Hall-Petrescu formula (see \cite[III.9.4]{hup} or \cite[page 37]{suz2}), we have
\begin{equation}
\label{hall-petrescu}
(fg)^{p^{\ell}}= f^{p^{\ell}}g^{p^{\ell}}c_2^{\binom{p^{\ell}}{2}}c_3^{\binom{p^{\ell}}{3}} \dots c_{p-1}^{\binom{p^{\ell}}{p-1}} \dots c_{p^{\ell}},
\end{equation}
where $c_i \in \gamma_i(\langle f,g \rangle)$.
Let $1\le i\le p^{\ell}$, and choose $r$ such that $p^r\le i<p^{r+1}$.
Then the binomial coefficient $\binom{p^{\ell}}{i}$ is divisible by $p^{\ell-r}$, by Kummer's Theorem.
Also,
\[
c_i \in \gamma_{p^r}(\langle f,g \rangle)
\le 
[\No_{z_k+1}, \No_{z_k-1}, \overset{p^r-1}{\ldots},\No_{z_k-1}]
=
\No_{2+p^r(z_k-1)+\frac{p^r-1}{p-1}},
\]
by using (\ref{comms Ni}).
Now, since
\[
2+p^r(z_k-1)+\frac{p^r-1}{p-1} = z_{r+k}+1,
\]
we get
\[
c_i ^{\binom{p^{\ell}}{i}} \in \No_{z_{r+k}+1}^{p^{\ell-r}} \le \No_{p^{\ell-r}(z_{r+k}+1)} \le \No_{z_{k+\ell}+1},
\]
by (\ref{powers Nzm}).
On the other hand,
\[
g^{p^{\ell}} \in \No_{z_k+1}^{p^{\ell}} \leq \No_{p^{\ell}(z_k+1)} \leq \No_{z_{k+\ell}+1},
\]
and we conclude from (\ref{hall-petrescu}) that $(fg)^{p^{\ell}} \equiv f^{p^{\ell}} \pmod{ \No_{z_{k+\ell}+1}}$.
\end{proof}

Thus if we want to know the $p^m$th powers in $\No/\No_{z_m+1}$ corresponding to all elements outside $\No'$,
it suffices to calculate that power for one specific element in every difference $\MM\smallsetminus \No'$, as $\MM$ runs over the $p+1$ maximal subgroups of $\No$.
These maximal subgroups are $\No_2$ and the subgroups $\MM_{\lambda}=\langle f_{\lambda} \rangle \No'$ for
all $\lambda\in\F_p$, where $f_{\lambda}$ is given by $f_{\lambda}(t)= t+t^2+ \lambda t^3$.
By \cite[Proposition 1.2]{klo}, the elements $a\in \MM_1\smallsetminus \No'$ and
$b\in \No_2\smallsetminus \No'$ given by $a(t)=t(1-t)^{-1}$ and $b(t)=t(1-2t)^{-1/2}$ are both of order $p$.
Thus all elements in $\MM_1\smallsetminus \No'$ and $\No_2\smallsetminus \No'$ have order at most $p^m$ in the quotient $\No/\No_{z_m+1}$.
Our next goal is to see that the situation is completely different in the maximal subgroups
$\MM_{\lambda}$ with $\lambda\ne 1$, and for that we need to know $f_{\lambda}^{p^m}$ modulo
$\No_{z_m+1}$.

Before we proceed, let us mention how one can calculate $p^m$th powers of elements in the Nottingham group.
Given $f\in\No$, we can form a unitriangular matrix $M$ with infinitely many rows and columns, by letting $M_{i,j}$ be the coefficient of $t^j$ in the power series $f(t^i)$, the image of $t^i$ under $f$.
Then we have the following result \cite[Lemma 5]{yor}.

\begin{lem}
Let $f\in\No$, and let $M$ be the matrix associated to $f$.
Then, for every $r\ge 1$, the coefficient of $t^n$ in the series $f^{p^r}(t)$ is
\begin{equation}
\label{coefficient fpm}
\sum_{\mathbf{i}=(i_0,\ldots,i_{\ell})} \, M_{i_0,i_1} M_{i_1,i_2} \ldots M_{i_{\ell-1},i_{\ell}},
\end{equation}
where $\ell=p^r$ and the tuples $\mathbf{i}=(i_0,\ldots,i_{\ell})$ in the sum are taken so that
$1=i_0<i_1<i_2<\cdots<i_{\ell-1}<i_{\ell}=n$.
\end{lem}

The next lemma, which describes how the $p$th power map behaves from one distinguished diamond to the next one, is crucial for the determination of Beauville structures in quotients of the Nottingham group.

\begin{lem}
Let $f\in\No$ be such that
\[
f(t)\equiv t+\lambda t^{z_{m-1}}+\mu t^{z_{m-1}+1} \pmod{t^{z_{m-1}+2}},
\]
where $m\ge 1$.
Then
\begin{equation}
\label{pth power}
f^p(t)
\equiv
\begin{cases}
t + \lambda^{p-1}(\lambda^2-\mu) t^{z_1} - \lambda^{p-2} (\lambda^2-\mu)^2 t^{z_1+1}
\hspace{-5pt} \pmod{t^{z_1+2}},
&
\text{if $m=1$,}
\\
t - \lambda^{p-1}\mu t^{z_m} - \lambda^{p-2}\mu^2 t^{z_m+1} \pmod{t^{z_m+2}},
&
\text{if $m>1$.}
\end{cases}
\end{equation}
\end{lem}

\begin{proof}
By Lemma \ref{all_elements}, we may assume that $f(t)=t+\lambda t^{z_{m-1}}+\mu t^{z_{m-1}+1}$.
According to the definition of the matrix $M$, we have
\[
f(t^i) = t^i + \sum_{j\ge 1} \, M_{i,i+j} t^{i+j}.
\]
By expanding the $i$th power in
\[
f(t^i) = f(t)^i = (t+t^{z_{m-1}}(\lambda+\mu t))^i,
\]
one readily obtains the following values of $M_{i,i+j}$ for $1\le j\le z_{m-1}+1$: if $m>1$ then we have
\begin{equation}
\label{M for m>1}
M_{i,i+j}
=
\begin{cases}
\lambda i,
&
\text{if $j=z_{m-1}-1$,}
\\
\mu i,
&
\text{if $j=z_{m-1}$,}
\\
0,
&
\text{if $1\le j<z_{m-1}-1$ or $j=z_{m-1}+1$,}
\end{cases}
\end{equation}
and if $m=1$ then
\begin{equation}
\label{M for m=1}
M_{i,i+j}
=
\begin{cases}
\lambda i,
&
\text{if $j=1$,}
\\[3pt]
\lambda^2 \binom{i}{2} + \mu i,
&
\text{if $j=2$,}
\\[3pt]
2\lambda\mu \binom{i}{2} + \lambda^3 \binom{i}{3},
&
\text{if $j=3$.}
\end{cases}
\end{equation}

Let us first assume that $m>1$.
We start by calculating the coefficient $\alpha$ of $t^{z_m}$ in $f^p(t)$.
To this purpose, we rely on formula (\ref{coefficient fpm}), applied with $n=z_m$.
By (\ref{M for m>1}), if the sum in (\ref{coefficient fpm}) corresponding to a vector $\mathbf{i}=(i_0,\ldots,i_p)$
is non-zero, we must have $i_{j+1}\ge i_j+z_{m-1}-1$ for every $j=0,\ldots,p-1$.
Thus $i_j\ge j(z_{m-1}-1)+1$ for $j=0,\ldots,p$.
Since $i_p=p(z_{m-1}-1)+2$, for some $k\in\{1,\ldots,p\}$ we must have $i_j=j(z_{m-1}-1)+1$ for $j=0,\ldots,k-1$
and $i_j=j(z_{m-1}-1)+2$ for $j=k,\ldots,p$.
Let us write, for simplicity, $q_j=j(z_{m-1}-1)+1$.
Then
\[
\alpha = \sum_{k=1}^p \, \alpha_k,
\]
where
\[
\alpha_k
=
\Big( \prod_{i=1}^{k-1} M_{q_{i-1},q_i} \Big)
\,
M_{q_{k-1},q_k+1}
\,
\Big( \hspace{-2pt} \prod_{i=k+1}^p M_{q_{i-1}+1,q_i+1} \Big).
\]
Now, by (\ref{M for m>1}),
\begin{gather*}
M_{q_{j-1},q_j} = \lambda q_{j-1} = \lambda j,
\quad
\text{for $j=1,\ldots,k-1$,}
\\[5pt]
M_{q_{k-1},q_k+1} = \mu q_{k-1} = \mu k,
\\
\intertext{and}
M_{q_{j-1}+1,q_j+1} = \lambda (q_{j-1}+1) = \lambda (j+1),
\quad
\text{for $j=k+1,\ldots,p$.}
\end{gather*}
Consequently
\[
\alpha_k = \lambda^{p-1} \mu \, k! (k+2)\ldots (p+1),
\]
and since this product contains the factor $p$ unless $k=p-1$, we finally get
\[
\alpha = \lambda^{p-1} \mu (p-1)! = -\lambda^{p-1} \mu,
\]
as desired.

The coefficient $\beta$ of $t^{z_m+1}$ in $f^p(t)$ can be obtained in a similar way.
Again, $\beta$ is a product of factors of the form $M_{i_{j-1},i_j}$, where $i_j-i_{j-1}=z_{m-1}-1$ except for two values
$k$ and $\ell$ for which $i_k-i_{k-1}=i_{\ell}-i_{\ell-1}=z_{m-1}$, or one value $r$ for which
$i_r-i_{r-1}=z_{m-1}+1$.
The latter case gives a zero product, since $M_{i,i+z_{m-1}+1}=0$ by (\ref{M for m>1}), and consequently
\[
\beta
=
\sum_{1\le k<\ell \le p}
\, \beta_{k,\ell},
\]
where
\begin{multline*}
\beta_{k,l}
=
\Big( \prod_{i=1}^{k-1} M_{q_{i-1},q_i} \Big)
\,
M_{q_{k-1},q_k+1}
\,
\Big( \hspace{-2pt} \prod_{i=k+1}^{\ell-1} M_{q_{i-1}+1,q_i+1} \Big)
\\
M_{q_{\ell-1}+1,q_\ell+2}
\,
\Big( \hspace{-2pt} \prod_{i=\ell+1}^p M_{q_{i-1}+2,q_i+2} \Big).
\end{multline*}
By (\ref{M for m>1}), we have
\[
\beta_{k,\ell} = \lambda^{p-2}\mu^2 \prod_{\substack{i=1\\ i\ne k+1,\,\ell+2}}^{p+2} \, i,
\]
which is $0$ unless $k=p-1$ and $\ell=p$, or $1\le k\le p-3$ and $\ell=p-2$.
Consequently
\[
\beta
=
\lambda^{p-2}\mu^2 \Big( (p-1)! (p+1) + \sum_{k=1}^{p-3} \, \frac{(p-1)! (p+1)(p+2)}{k+1} \Big)
=
-\lambda^{p-2}\mu^2,
\]
where the last equality follows from the fact that
\[
\sum_{i=1}^{p-1} \, \frac{(p-1)!}{i} \equiv 0 \pmod p
\]
for odd $p$, since it coincides with the sum of the inverses of all elements of $\F_p^{\times}$.
This completes the proof when $m>1$.

Finally, the case $m=1$ can be dealt with in a similar way, and the details are left to the reader.
\end{proof}

\begin{cor}
\label{p^mth_powers}
For every $\lambda\in\F_p$ and $m\ge 1$, we have
\[
f_{\lambda}^{p^m}(t) \equiv t+ (1-\lambda)^m t^{z_m}- (1- \lambda)^{m+1}t^{z_m+1} \pmod{t^{z_m+2}}.
\]
\end{cor}

On several occasions, we will find ourselves working in a quotient of the form
$\No/\No_{z_m+1}$, which for simplicity we will call $G$.
In that case, we will systematically write $N_k$ for $\No_k/\No_{z_m+1}$, for
$1\le k\le z_m+1$, and $M_{\lambda}$ instead of $\MM_{\lambda}/\No_{z_m+1}$,
for every $\lambda\in\F_p$.
We will use this notation without further reference.

\begin{cor}
\label{pmth powers}
If $G=\No/\No_{z_m+1}$, then for $\lambda\in\F_p$, $\lambda\ne 1$, the power subgroups $M_{\lambda}^{p^m}$ are all different and of order $p$, contained in $N_{z_m-1}$.
In particular, all elements of $M_{\lambda}\smallsetminus G'$ are of order $p^{m+1}$ for $\lambda\ne 1$.
\end{cor}

The following result, which gives a sufficient condition to lift a Beauville structure from a quotient group, is Lemma 4.2 in \cite{FJ}.

\begin{lem}
\label{lifting structure}
Let $G$ be a finite group and let $\{x_1,y_1\}$ and $\{x_2,y_2\}$ be two sets of generators of $G$.
Assume that, for a given $N\trianglelefteq G$, the following hold:
\begin{enumerate}
\item
$\{x_1N,y_1N\}$ and $\{x_2N,y_2N\}$ is a Beauville structure for $G/N$.
\item
$o(g)=o(gN)$ for every $g\in\{x_1,y_1,x_1y_1\}$.
\end{enumerate}
Then $\{x_1,y_1\}$ and $\{x_2,y_2\}$ is a Beauville structure for $G$.
\end{lem}

We can now begin to determine which quotients of the Nottingham group are Beauville groups.
Firstly, we consider quotients of the form $\No/\No_k$.
We deal separately with the cases $p>3$ and $p=3$.

\begin{thm}
\label{nott p>3}
If $p\ge 5$ then a quotient $\No/\No_k$ is a Beauville group if and only if $k\ge 3$ and $k\ne z_m$ for all $m\ge 1$. 
\end{thm}

\begin{proof}
First of all, we show that $G=\No/\No_{z_m+1}$ is a Beauville group for all $m\ge 1$.
Let $u$ and $v$ be the images in $G$ of the automorphisms $a$ and $b$ which were defined after Lemma \ref{all_elements}.
Then $\{u,v\}$ and $\{uv^2,uv^4\}$ are both systems of generators of $G$, and we claim that they yield a Beauville structure for $G$.
If $X=\{u,v,uv\}$ and $Y=\{uv^2,uv^4,uv^2uv^4\}$, we have to see that
\begin{equation}
\label{check}
\langle x^g \rangle \cap \langle y^h \rangle=1
\end{equation}
for all $x\in X$, $y\in Y$, and $g,h\in G$.
Observe that $\langle x\Phi(G) \rangle$ and $\langle y\Phi(G) \rangle$ have trivial intersection for every $x\in X$ and
$y\in Y$, since $a$ and $b$ are linearly independent modulo $\Phi(G)$.
As a consequence, $x^g$ and $y^h$ lie in different maximal subgroups of $G$ in every case.

Assume first that $x=u$ or $v$, which are elements of order $p$.
If (\ref{check}) does not hold, then $\langle x^g \rangle \subseteq \langle y^h \rangle$, and consequently $\langle x\Phi(G) \rangle=\langle y\Phi(G) \rangle$, which is a contradiction.
Thus we assume that $x=uv$.
Now, $uv$ and all elements $y\in Y$ lie in $M_{\lambda}\smallsetminus G'$ for some
$\lambda\in\F_p$, $\lambda\ne 1$, and so they are all of order $p^{m+1}$, by Corollary \ref{pmth powers}.
If (\ref{check}) does not hold, then
\[
\langle (x^g)^{p^m} \rangle = \langle (y^h)^{p^m} \rangle
\]
and, again by Corollary \ref{pmth powers}, $x^g,y^h\in M_{\lambda}$ for some
$\lambda$.
This is a contradiction, and we thus complete the proof that $G$ is a Beauville group.

Let us now consider a quotient $\No/\No_k$ with $k\ge 3$ and $k\ne z_m$ for all
$m\ge 1$.
Then either $3\le k\le p+1$ or $z_m+1\le k\le z_{m+1}-1$ for some $m\ge 1$.
In the former case, $\No/\No_3\cong C_p\times C_p$ is a Beauville group, since
$p\ge 5$, and $\exp \No/\No_k=\exp \No/\No_3=p$.
Thus $\No/\No_k$ is a Beauville group by Lemma \ref{lifting structure}.
In the latter case, $\No/\No_{z_m+1}$ has the Beauville structure shown in the previous paragraph, whose first set of generators is $\{a\No_{z_m+1},b\No_{z_m+1}\}$.
Since $o(a \No_k)=o(b \No_k)=p$ and $o(ab \No_k)=p^{m+1}$, since
$(ab)^{p^{m+1}}\in \No_{z_{m+1}-1}$ by Corollary \ref{pmth powers}, we can again apply Lemma \ref{lifting structure}, and $\No/\No_k$ is a Beauville group also in this case.

Let us finally see that $\No/\No_{z_m}$ is not a Beauville group for $m\ge 1$.
By Lemma \ref{all_elements} and Corollary \ref{p^mth_powers}, we know that all elements in $\MM_{\lambda}/\No_{z_m}\smallsetminus \No'/\No_{z_m}$ are of order $p^{m+1}$ for $\lambda\ne 1$.
Since $\exp \No/\No_{z_m}=p^{m+1}$, it follows that condition (i) of Proposition \ref{negative} is fulfilled.
On the other hand, by (\ref{powers Ni}) we have $\No^{p^m}\le \No_{z_m-1}$, and so
$(\No/\No_{z_m})^{p^m}$ has order $p$.
Thus also condition (ii) of Proposition \ref{negative} holds, and we conclude that
$\No/\No_{z_m}$ is not a Beauville group.
\end{proof}

In order to deal with the prime $3$, we need two more lemmas.

\begin{lem}
\label{intersection_2}
Let $G$ be a finite $p$-group and let $x \in G \setminus \Phi(G)$ be an element of order $p$.
If $t \in \Phi(G)\smallsetminus \{[x,g] \mid g \in G \}$ then 
\[
\Big(\bigcup_{g\in G} {\langle x\rangle}^g  \Big)
\bigcap
\Big(\bigcup_{g\in G} {\langle xt\rangle}^g \Big)= 1.
\]
\end{lem}

\begin{proof}
We assume that $h= (x^i)^{g_1}=((xt)^j)^{g_2}$ for some $i,j \in \mathbb{Z}$ and $g_1, g_2 \in G$, and prove that
$h=1$.
Since $t\in\Phi(G)$, we have $x^i\Phi(G)=x^j\Phi(G)$, and so $i\equiv j \pmod{p}$.
Then $h=(x^{g_1})^j=((xt)^{g_2})^j$, since $x$ is of order $p$.
If $p\mid j$ then we are done.
If $p\nmid j$ then, since $G$ is a $p$-group, $x^{g_1}=(xt)^{g_2}$, and consequently $t= [x,g_1{g_2}^{-1} ]$, which is a contradiction.
\end{proof}

In the following lemma, we need a result of Klopsch \cite[formula (3.4)]{klo} regarding the centralizers of elements of order $p$ of the Nottingham group in some quotients
$\No/\No_k$.
More specifically, if $f\in \No_k\smallsetminus \No_{k+1}$ then for every
$\ell=k+1+pn$ with $n\in\N$, we have
\begin{equation}
\label{centralizer}
C_{\No/\No_{\ell}}(f\No_{\ell}) = C_{\No}(f)\No_{\ell-k}/\No_{\ell}.
\end{equation}

\begin{lem}
\label{not covered}
Let $p=3$ and $m\ge 1$, and put $G=\No/\No_{z_m+1}$ and
$N_k=\No_k/\No_{z_m+1}$ for all $k\ge 1$.
If $u$ and $v$ are the images of $a$ and $b$ in $G$, respectively, then the sets
$\{[u,g]\mid g\in G\}$ and $\{[v,g]\mid g\in G\}$ do not cover $N_{z_m-1}$.
\end{lem}

\begin{proof}
We first consider the element $u$, for which we prove that
\[
\{[u,g]\mid g\in G\} \cap N_{z_m}=1.
\]
To see this, assume that $[u,g]\in N_{z_m}$.
Since $a\in\No_1\smallsetminus \No_2$ is of order $p$, (\ref{centralizer}) yields
\[
C_{\No/\No_{z_m}}(a\No_{z_m}) = C_{\No}(a)\No_{z_m-1}/\No_{z_m}.
\]
Thus we can write $g=ch$, with $[u,c]=1$ and $h\in N_{z_m-1}$.
It follows that $[u,g]=[u,h]\in [G,N_{z_m-1}]=1$, since
$N_{z_m-1}$ is central in $G$ (it corresponds to the diamond 
$\No_{z_m-1}/\No_{z_m+1}$ in $\No$).

We prove the result for the element $v$ by showing that, whenever
$[v,g]\in N_{z_m-1}$, we actually have $[v,g]\in N_{z_m}$.
Since $b\in\No_2\smallsetminus \No_3$ is of order $p$, we have
\[
C_{\No/\No_{z_m-2}}(b\No_{z_m-2}) = C_{\No}(b)\No_{z_m-4}/\No_{z_m-2}.
\]
Thus we can write $g=ch$ with $[v,c]=1$ and $h\in N_{z_m-4}$, and consequently
$[v,g]=[v,h]$.
Now, since $[N_2,N_{z_m-3}]=N_{z_m}$ and the commutator of 
$v\in N_2\smallsetminus N_3$ with an element of
$N_{z_m-4}\smallsetminus N_{z_m-3}$ lies in $N_{z_m-2}\smallsetminus N_{z_m-1}$,
we conclude that $[v,g]\in N_{z_m}$, as desired.
\end{proof}

\begin{thm}
\label{nott p=3}
If $p=3$ then a quotient $\No/\No_k$ is a Beauville group if and only if $k\ge 6$ and
$k\ne z_m$ for all $m\ge 1$. 
\end{thm}

\begin{proof}
Since the smallest order of a Beauville $3$-group is $3^5$, the quotient $\No/\No_k$ can only be a Beauville $3$-group if
$k\ge 6$; note that $6$ is the same as $z_1+1$ in this case.
Now, by arguing as in the proof of Theorem \ref{nott p>3}, it suffices to see that
$G=\No/\No_{z_m+1}$ is a Beauville group for every $m\ge 1$.
Recall that $N_k=\No_k/\No_{z_m+1}$ and $M_{\lambda}=\MM_{\lambda}/\No_{z_m+1}$.

Let $u$ and $v$ be the images of $a$ and $b$ in $G$, respectively.
By Lemma \ref{not covered}, there exist $w,z\in N_{z_m-1}$ such that
$w\not\in \{[u,g]\mid g\in G\}$ and $z\not\in \{[v,g]\mid g\in G\}$.
Observe that $w$ and $z$ are central elements of order $p$ in $G$.

We claim that $\{u,v\}$ and $\{(uw)^{-1},vz\}$ form a Beauville structure in $G$.
Let $X=\{u,v,uv\}$ and $Y=\{(uw)^{-1},vz,u^{-1}vw^{-1}z\}$.
Assume first that $x\in X$ is of order $p$, and let $y\in Y$.
If $\langle x\Phi(G) \rangle \ne \langle y\Phi(G) \rangle$ in $G/\Phi(G)$, then we get
$\langle x\rangle^g \cap \langle y \rangle^h=1$ for every $g,h\in G$, as in the proof of Theorem \ref{nott p>3}.
Otherwise, we are in one of the following two cases: $x=u$ and $y=(uw)^{-1}$, or
$x=v$ and $y=vz$.
Then the condition $\langle x \rangle^g \cap \langle y \rangle^h=1$ follows by combining Lemma \ref {intersection_2} and Lemma \ref{not covered}.

Since the same argument applies when $y\in Y$ is of order $p$, we are only left with the case when $x=uv$ and $y=u^{-1}vw^{-1}z$.
Now $x$ and $y$ lie in two different maximal subgroups, which are also different from $N_2$ and $M_1$.
By Corollary \ref{pmth powers}, both $x$ and $y$ are of order $p^{m+1}$ and
$\langle x^{p^m} \rangle \ne \langle y^{p^m} \rangle$.
Since $x^{p^m},y^{p^m}\in N_{z_m-1}$ are central in $G$, it follows that
$\langle x \rangle^g \cap \langle y \rangle^h=1$ for all $g,h\in G$ also in this case.
This completes the proof.
\end{proof}

Finally, we analyse the quotients of the form $\No/\WW$, where $\WW$ is an intermediate subgroup in a diamond
$\No_{kp+1}/\No_{kp+3}$, i.e. $\No_{kp+3}<\WW<\No_{kp+1}$.

\begin{thm}
\label{thm:quotient by intermediate}
Let $\WW$ be an intermediate subgroup in a diamond of the Nottingham group.
Then:
\begin{enumerate}
\item
If the diamond that contains $\WW$ is not distinguished, then $\No/\WW$ is a Beauville group.
\item
If the diamond that contains $\WW$ is distinguished, say $\No_{z_m+1}<\WW<\No_{z_m-1}$, then $\No/\WW$ is a Beauville group if and only if $m\ge 1$ or $m\ge 2$, according as $p>3$ or $p=3$, and furthermore $\WW \ne \No_{z_m}, \langle e,\No_{z_m+1} \rangle$, where $e$ is the automorphism given by $e(t)=t+t^{z_m}$.
\end{enumerate}
\end{thm}

\begin{proof}
(i)
Let $\No_{kp+3}<\WW<\No_{kp+1}$, and choose $m$ as large as possible such that $z_m+1<kp+3$.
By Theorems \ref{nott p>3} and \ref{nott p=3}, we have a Beauville structure in $\No/\No_{z_m+1}$ in which one of the sets of generators is
$\{a\No_{z_m+1},b\No_{z_m+1}\}$.
Now one can readily check that $a$, $b$ and $ab$ have the same order modulo $\WW$ and modulo $\No_{z_m+1}$, namely $p$, $p$ and $p^{m+1}$.
Hence $\No/\WW$ is a Beauville group by Lemma \ref{lifting structure}.

(ii)
By looking at the order of $\No/\WW$, it is clear that the conditions $m\ge 1$ if $p>3$ and $m\ge 2$ if $p=3$ are necessary for $\No/\WW$ to be a Beauville group.
On the other hand, the $p+1$ intermediate subgroups between $\No_{z_m+1}$ and $\No_{z_m-1}$ are $\No_{z_m}$ and the subgroups
$\WW_{\alpha}=\langle e_{\alpha}, \No_{z_m+1} \rangle$, where $\alpha\in\F_p$ and $e_{\alpha}(t)=t+t^{z_m}+\alpha t^{z_m+1}$.
We already know that $\No/\No_{z_m}$ is not a Beauville group, and the same argument shows that neither $\No/\WW_0$ is a Beauville group.
Thus it suffices to consider the case where $\WW=\WW_{\alpha}$ for some $\alpha\ne 0$, and prove that $\No/\WW$ is a Beauville group in that case.
If we define $f_{\lambda}$ as above by means of $f_{\lambda}(t)=t+t^2+\lambda t^3$, then Corollary \ref{p^mth_powers} yields that
\[
f_{1+\alpha}^{p^m} \equiv e_{\alpha}^{(-\alpha)^m}  \pmod{\No_{z_m+1}},
\]
and consequently $\WW=\langle f_{1+\alpha}^{p^m}, \No_{z_m+1} \rangle$.
Now observe that $f_{1+\alpha}\equiv ab^{\alpha} \pmod{\No_3}$ which, according to Lemma \ref{all_elements}, implies that
\[
f_{1+\alpha}^{p^m}\equiv (ab^{\alpha})^{p^m} \pmod{\No_{z_m+1}}.
\]
Hence $\WW=\langle (ab^{\alpha})^{p^m}, \No_{z_m+1} \rangle$.
In particular, the order of $ab^{\alpha}$ modulo $\WW$ is $p^m$.

Now, since $m\ge 1$ if $p>3$ and $m\ge 2$ if $p=3$, $\No/\No_{z_{m-1}+1}$ has a Beauville structure with $\{a\No_{z_{m-1}+1},b\No_{z_{m-1}+1}\}$ as one of the generating sets.
One can similarly see that there is a Beauville structure with $\{a\No_{z_{m-1}+1},b^{\alpha}\No_{z_{m-1}+1}\}$.
Since $o(a\No_{z_{m-1}+1})=o(a\WW)=p$, $o(b\No_{z_{m-1}+1})=o(b\WW)=p$ and
$o(ab^{\alpha}\No_{z_{m-1}+1})=o(ab^{\alpha}\WW)=p^m$, we can apply Lemma \ref{lifting structure} to conclude that
$\No/\WW$ has a Beauville structure, as desired.
\end{proof}

At this point, Theorem B is straightforward.

\begin{proof}[Proof of Theorem B]
Put $n_0=2$ or $5$ according as $p>3$ or $p=3$, as in the statement of Theorem B.
Then Theorems \ref{nott p>3} and \ref{nott p=3} provide Beauville groups which are quotients of the Nottingham group of order $p^n$ for every $n\ge n_0$, with the only exception of the values $n=p^m+\cdots+p+1$, where $m\ge 1$ if $p>3$ and $m\ge 2$ if $p=3$.
Now, according to the previous theorem, the missing orders can be obtained by factoring $\No$ with an adequate intermediate subgroup in the distinguished diamond $\No_{z_m-1}/\No_{z_m+1}$.
\end{proof}

We close the paper by showing that condition (i) in Proposition \ref{negative} cannot be relaxed.
More specifically, given a $2$-generator finite $p$-group $G$ of exponent $p^e$ in which $\Omega_{\{e-1\}}(G)$ is contained in the union of three maximal subgroups and
$|G^{p^{e-1}}|=p$, it may well happen that $G$ is a Beauville group.
To this end, consider an intermediate subgroup $\No_{z_m+1}<\WW<\No_{z_m-1}$ in a distinguished diamond of the Nottingham group, where $m\ge 1$ or $m\ge 2$, according as $p>3$ or $p=3$.
If $\WW \ne \No_{z_m}, \langle e,\No_{z_m+1} \rangle$ then $G=\No/\WW$ is a Beauville group by Theorem \ref{thm:quotient by intermediate}.
Also, as indicated in the proof of that theorem, we have $\WW=\langle (ab^{\alpha})^{p^m},\No_{z_m+1} \rangle$ for some $\alpha\ne 0$ in $\F_p$.
It then follows from Corollary \ref{pmth powers} that $\exp G=p^{m+1}$ and that $\Omega_{\{m\}}(G)$ is contained in the three maximal subgroups of $G$ that contain the images of $a$, $b$ and $ab^{\alpha}$.

\end{document}